\documentclass[11pt]{article}
\usepackage{amsmath,amssymb,amsfonts}
\date{}
\newtheorem{theorem}{Theorem}[section]

\newtheorem{definition}[theorem]{Definition}
\newtheorem{lemma}[theorem]{Lemma}

\newtheorem{question}[theorem]{Question}

\newenvironment{proof}[1][Proof]{\noindent\textbf{#1.} }{\ \rule{0.5em}{0.5em}}

\numberwithin{equation}{section}
\begin{document}

\title{On double Danielewski varieties \footnote{This work was supported by the NSFC (12371020, 12171194), NSFJP (20220101019JC, 20210101469JC).}}
\author{Xiaosong Sun\footnote{Corresponding author. E-mail: sunxs@jlu.edu.cn}, Shuai Zeng
\\
School of Mathematics\\
Jilin University, Changchun 130012, China\\} \maketitle

\begin{abstract}
In this paper, we study the double Danielewski varieties which arose from the research on the classical Cancellation Problem. We describe the Makar-Limanov invariant and locally nilpotent derivations of these varieties. And in a subsequent paper we will describe the automorphisms groups of the varieties and verify that the varieties are counterexamples to the Cancellation Problem.
\end{abstract}

\section{Introduction}\label{xxsec1}

The Cancellation Problem is a central problem in affine algebraic geometry (cf. \cite{eakin73}). For a ring $R$, we denote by $R^{[n]}$ the
polynomial ring over $R$ in $n$ variables.

\begin{question}[Cancellation Problem] \label{cp}
	Let $k$ be a field. If $F$ and $G$ are two finitely generated $k$-algebras such that $F^{[1]}\cong_{k}G^{[1]}$, then does it follow that $F\cong_{k} G$?
\end{question}

In 1989, Danielewski \cite{dan89} constructed a family of two-dimensional affine domains over the field of complex numbers $\mathbb{C}$ which are counterexamples of the above question. For a non-constant polynomial $P(z)$ with distinct roots, Danielewski considered the affine surface $S_n$ defined by $x^{n}y-P(z)=0$ in $\mathbb{A}_{k}^{3}$. Write $A(S_n)$ for the coordinate ring of $S_n$. It is known that for any pair $(m,n)$ with $m\neq n,$
$S_m\ncong_k S_n$ but $S_m\times \mathbb{A}_k^1\cong S_n\times \mathbb{A}_k^1$, equivalently,
$A(S_{m})\ncong_{k} A(S_{n}) $ but $A(S_{m})^{[1]}\cong_k A(S_{n})^{[1]}$ (cf. \cite{fie94}).

In 2001, Makar-Limanov \cite{makar01} determined the automorphism group of $S_{n}$ by computing the Makar-Limanov invariant of $S_{n}$.
The Makar-Limanov invariant (see Definition \ref{def-ml} below) is a powerful tool for investigating affine varieties, which was  introduced by Makar-Limanov \cite{makar96} in 1996. When the field $k$ is of characteristic zero, the Makar-Limanov invariant  of an affine variety $X$ over $k$ is the intersection of all the kernels of locally nilpotent derivations of the coordinate ring $A(X)$.

Nowadays, the term Danielewski surface refers to the surfaces $S_n$ above or refers to  more general surfaces $S_{P, n}$ defined by equations of the form $x^{n}y-P(x,z)=0$ with certain conditions on $P(x, z)$. These surfaces were studied extensively in literature in algebraic or geometric settings. For example, Crachiola \cite{cra} studied the case where $P(x,z)=y^2+\sigma(x)y$ with $n\geq 2$ and $\sigma(0)\neq 0$; and in 2009, Dubouloz and Poloni \cite{dub09} considered the case where $n\geq 2$ and $P(x,z)\in k[x,z]$ is a polynomial such that $P(0, z)$ splits with $r\geq 2$ simple roots in $k$, and they computed the automorphism groups of the surfaces $S_{P, n}$ in this case.

In 2019, Gupta and Sen \cite{gupta2019double} considered the so-called double Danielewski surfaces defined  by equations of the form $\{x^{r}y-P(x,z)=0, ~x^{s}t-Q(x,y,z)=0\}$ in $\mathbb{A}_{k}^{4}$, where $r\geq 2, s\geq 2, P(x,z)$ is monic in $z$ and $Q(x, y,z)$ is monic in $y$ with $\deg_{z}P(x,z)\geq 2, \deg_{y}Q(x,y,z)\geq 2$; they computed the Makar-Limanov invariant and isomorphism class of the surfaces and verified that they are counterexamples to the Cancellation Problem. And subsequently, they described in \cite{gupta2020locally}  all the locally nilpotent derivations of the double Danielewski surfaces when $k$ is of characteristic zero.

In 2007, Dubouloz \cite{dub07} considered a higher-dimensional analogue of Dan\-iel\-ewski surfaces. He constructed the Danielewski varieties defined by equation\[x_{1}^{d_{1}}\cdots x_{n}^{d_{n}}y-P(x_{1},\ldots,x_{n},z)=0\] in $\mathbb{A}_{\mathbb{C}}^{n+2}$, where $P(x_{1},\ldots, x_{n},z)$ is monic in $z$ and $\deg_{z}P>1$, and he proved that they are counterexamples to the Cancellation Problem. In 2023,  Ghosh and Gupta \cite{gho23} extended the result of Dubouloz to arbitrary field.

In this paper, we study a higher-dimensional analogue of the double Danielewski surfaces. More precisely, we consider a family of affine varieties $L_{[d],[e]}$ of arbitrary dimension greater or equal to 2 over any field $k$, defined by a pair of equations
\[\{\underline{x}^{[d]}y-P(\underline{x},z)=0,~ \underline{x}^{[e]}t-Q(\underline{x},y,z)=0\} \ \text{in}\  \mathbb{A}^{n+3}_{k},\]
where $\underline{x}=(x_1,x_2,\ldots,x_n)$, $[d],[e]\in \mathbb{Z}^{n}_{\geq 1}, n\in \mathbb{N}\backslash\{0\}, P(\underline{x}, z)$ is monic in $z$ and $Q(\underline{x},y,z)$ is monic in $y$ with $r:=\deg_{z}P(\underline{x},z)\geq 2,$ and $s:=\deg_{y}Q(\underline{x},y,z)\geq 2$. We
call them \emph{double Danielewski  varieties}. (When $r=1$ or $s=1$, $L_{[d],[e]}$ is a Danielewski variety). In Section 3, we determine the Makar-Limanov invariant of the double Danielewski varieties (Theorem \ref{thm: ML-invariantofL}). And in Section 4 we describe the locally nilpotent derivations of  these varieties when the field $k$ is of characteristic zero (Theorem \ref{thm: computeof LND}). And in a subsequent paper we will describe the automorphisms groups of the varieties and verify that they are counterexamples to the Cancellation Problem.

\section{Preliminaries}
In this section, we recall some notions and basic facts concerning exponential maps, Makar-Limanov invariant and associated graded rings.

Recall that, for a ring $R$, the notation $R^{[n]}$ means the polynomial ring over $R$ in $n$ variables.

\begin{definition} Let $R$ be a $k$-algebra over a field $k$ and let $\phi : R\rightarrow R^{[1]}$ be a $k$-algebra homomorphism. When $R^{[1]}=R[W]$ for an indeterminate $W$ we write $\phi_W$ for $\phi$. We say that $\phi$ is an exponential  map on $R$ if $\phi$ satisfies  the following properties:
	\begin{enumerate}
		\item $\epsilon_0\phi_{W}=\textup{id}_R$, where $\epsilon_0: R[W]\rightarrow R$ is the evaluation map at $W=0$.
		\item  $\phi_{T}\phi_{W}=\phi_{W+T}$, where $\phi_{T}: R\rightarrow R[T]$ is extended to a homomorphism $\phi_{T}: R[W]\rightarrow R[T,W]$ by setting $\phi_{T}(W)=W$.
	\end{enumerate}
\end{definition}

For an exponential map $\phi$ on $R$, the subring $R^{\phi}:=\{a\in R\ |\ \phi(a)=a\}$ of $R$ is called the invariants ring of $\phi$.
We say that $\phi$ is nontrivial if $R^{\phi}\neq R$. We denote by $\mathrm{EXP}(R)$ the set of all exponential maps on $R$.

\begin{definition} \label{def-ml} The Makar-Limanov invariant (ML invariant for short) of a ring $R$ is the subring defined by \[\mathrm{ML}(R):= \bigcap_{\phi\in \mathrm{EXP}(R)}R^{\phi}.\]
When $\mathrm{ML}(R)=R$, we say that $R$ is a rigid ring.
\end{definition}

The following are some well-known properties of exponential maps on an affine domain.

\begin{lemma}\cite{gupta2014zariski}\label{lem:expproperty} Let $R$ be an affine domain over a field $k$. Suppose that there exists a nontrivial exponential map $\phi$ on $R$. Then the following statements hold:
	
	\begin{enumerate}
		\item\label{factorialproperty} $R^{\phi}$ is factorially closed in $R$, i.e., given $f,g\in R\backslash \{0\}$ the condition $fg\in R^{\phi} $ implies $f\in R^{\phi}$ and $g\in R^{\phi}$.
		\item\label{algebraicclosedproperty} $R^{\phi} $ is algebraically closed in $R$.
		\item\label{localsliceproperty}  If $x\in R\backslash R^\phi$ is such that $\deg_{W}\phi_W(x)$ is of minimal positive degree, and $c$ is the leading coefficient of $W$ in $\phi_W(x)$, then $c\in R^{\phi}$ and $R[c^{-1}]=R^{\phi}[c^{-1}][x]$.
		\item\label{invariantdimension}  $\textup{trdeg}_{k}(R^{\phi})= \textup{trdeg}_{k}(R)-1$.
		\item \label{1dimensionpolynomialcrieterion} If $\textup{trdeg}_{k}(R)=1$ and $\tilde{k}$ is the algebraic closure of $k$ in $R$, then $R=\tilde{k}^{[1]}$ and $R^{\phi}= \tilde{k}$.
		\item  For any multiplicative subset $S$ of $R^{\phi}\backslash\{0\}$, $\phi$ extends to a nontrivial exponential map $S^{-1}\phi$ on $S^{-1}R$ by setting $(S^{-1}\phi)(a/s):= \phi(a)/s$ for $a\in R, s\in S$, and the invariants ring of $S^{-1}\phi$ is $S^{-1}(R^{\phi})$.
		
	\end{enumerate}
\end{lemma}

Below we recall some notions and facts about filtration and associated graded ring.

\begin{definition}\label{def: grading} 
Let $R$ be a ring. A $\mathbb{Z}$-grading of $R$ is a family of subgroups $\{R_{n}\}_{n\in \mathbb{Z}}$ of $(R, +)$ such that:
	\begin{enumerate}
		\item $R= \oplus_{n\in \mathbb{Z}}R_{n}$.
		\item $R_{n}R_{m}\subseteq R_{n+m}$ for all $n, m \in \mathbb{Z}$.
	\end{enumerate}
\end{definition}

\begin{definition}
	Let $R$ be an affine domain over a field $k$. A family of $k$-linear subspaces $\{R_{n}\}_{n\in \mathbb{Z}}$ of $R$ is said to be a $\mathbb{Z}$-filtration if
	\begin{enumerate}
		\item[1.] $R_{n}\subseteq R_{n+1}$ for all $n\in \mathbb{Z}$;
		\item[2.] $R=\bigcup_{n\in \mathbb{Z}}R_{n}$,
	\end{enumerate}
and it is called a proper $\mathbb{Z}$-filtration if it has the following further properties
	\begin{enumerate}
			\item[3.] $\bigcap_{n\in \mathbb{Z}}R_{n}=\{0\}$;
		\item[4.] $(R_{n}\backslash R_{n-1})(R_{m}\backslash R_{m-1})\subseteq R_{n+m}\backslash R_{n+m-1}$ for all $n, m\in \mathbb{Z}$.
	\end{enumerate}	
\end{definition}	

For an affine domain $R$ over a field $k$, a proper $\mathbb{Z}$-filtration $\{R_{n}\}_{n\in \mathbb{Z}}$ of $R$ is called \emph{admissible} if there exists a finite generating set $\Lambda$ of $R$ such that, for each $n$ and each element $r\in R_{n}$, we can write $r$ as a finite sum of monomials of $\Lambda$ and all these monomials are in $R_{n}$.

It is well-known that a proper $\mathbb{Z}$-filtration on $R$ induces an \emph{associated $\mathbb{Z}$-graded domain}
\[\textup{Gr}(R):= \oplus_{i}R_{i}/R_{i-1},\]
and an associated map \[\rho: R\rightarrow \textup{Gr}(R)\ \text{where}\ \rho(a)=a+R_{n-1} \ \text{for any}\ a\in R_{n}\backslash R_{n-1}.\]

An exponential map $\phi$ on a graded ring $R$ is called \emph{homogeneous} if $\phi_W:R\rightarrow R[W]$ is homogeneous when the grading of $R[W]$ is induced by $R$ and by setting $W$ homogeneous; whence $R^{\phi}$ is a graded subring of $R$.

The following result is on the relation of between the invariants ring of $R$ and that of $\textup{Gr}(R)$.
\begin{theorem}\label{thm: invariantofhomogeneneousexp}\cite[Proposition 2.2]{derksen2001newton}
	Let $R$ be an affine domain over a field $k$ with an admissible proper $\mathbb{Z}$-filtration and $\textup{Gr}(R)$ the associated $\mathbb{Z}$-graded domain. Let $\phi$ be a nontrivial exponential map on $R$. Then $\phi$ induces a nontrivial homogeneous exponential map $\overline{\phi}$ on $\textup{Gr}(R)$ such that $\rho(R^{\phi})\subseteq \textup{Gr}(R)^{\overline{\phi}}$.
\end{theorem}

\section{ML-invariant of double Danielewski varieties}

In this section, we determine the ML-invariant of double Danielewski varieties over any field $k$.

The coordinate ring of a double Danielewski variety is of the form
\begin{equation}\label{eq: A}
	A(L_{[d],[e]})=\frac{k[\underline{X},Y,Z,T]}{(\underline{X}^{[d]}Y - P(\underline{X},Z), \underline{X}^{[e]}T - Q(\underline{X},Y,Z))},
\end{equation}
where $\underline{X}:=(X_1,\ldots,X_n)$, $Y, Z, T$ are independent variables, $n\geq 1$, $[d]=(d_1,\ldots,d_n)\in \mathbb{Z}_{\geq 1}^{n}$, $[e]=(e_1,\ldots,e_n)\in \mathbb{Z}_{\geq 1}^{n}$, $\underline{X}^{[d]}=X_1^{d_1}\cdots X_{n}^{d_{n}}$, $\underline{X}^{[e]}=X_1^{e_1}\cdots X_{n}^{e_{n}}$, $P(\underline{X},Z)\in k[\underline{X},Z]$ is monic in $Z$ and $Q(\underline{X}, Y, Z)\in k[\underline{X}, Y, Z]$ is monic in $Y$ with $r:=\deg_{Z}P(\underline{X}, Z)\geq 2$ and $s:=\deg_{Y}Q(\underline{X}, Y, Z)\geq 2$.

In what follows, we write $\underline{x}, y, z$ and $t$ for the images of $\underline{X}, Y, Z, T $ in $A(L_{[d],[e]})$ respectively.

We first give some lemmas which will be used to verify the integrity of $A(L_{[d],[e]})$.
\begin{lemma}\label{lem:zero-divisor}\cite[Lemma 3.1]{gupta2019double}
	Let $R$ be a  domain and $a, b \in R\backslash\{0\}$. If $a$ is not a zero-divisor on $R/(b)$, then $b^n$ is not a zero-divisor on $R/(a)$ for any integer $n\geq 1$.
\end{lemma}
\begin{lemma}\label{lem: gen zero-divisor}
	Let $R$ be a domain and $a, b_{1},\ldots, b_{n}\in R\backslash\{0\}$. Suppose that $a$ is not a zero divisor on $R/(b_{1},\ldots,b_{n})$, and let $\underline{b}=(b_1,\ldots,b_n)$. Then $\underline{b}^{[m]}$ is not a zero-divisor on $R/(a)$ for any $[m]=(m_1,\ldots,m_n)\in \mathbb{Z}_{\geq 1}^{n}$,
\end{lemma}
\begin{proof}
Suppose, by contrary, that $\underline{b}^{[m]}$ is a zero-divisor on $R/(a)$ for some $[m]\in \mathbb{Z}_{\ge 1}^{n}$. Then there exists some $i$ such that $b_{i}$ is a zero-divisor of $R/(a)$, which contradicts Lemma \ref{lem:zero-divisor}.
\end{proof}

\begin{lemma}\label{lem: integrity}\cite[Lemma 2.4]{dutta2012some}
	Let $R$ be a domain and $a, b\in R\backslash\{0\}$. If $b$ is not a zero-divisor on $R/(a)$, then the ring $\frac{R[T]}{bT-a}$ is a  domain.
\end{lemma}

\begin{lemma} \label{lem: domain}
	$A(L_{[d],[e]})$ is a  domain.
\end{lemma}
\begin{proof}
	Let $R=\frac{k[\underline{X}, Y, Z]}{(\underline{X}^{[d]}Y- P(\underline{X},Z))} $. Since $\underline{X}^{[d]}Y-P(\underline{X},Z)$ is linear in $Y$,  and $\underline{X}^{(1,\ldots,1)}$ and $P(\underline{X}, Z)$ ($P$ is monic in $Z$) are coprime in $k[\underline{X}, Z]$, we have $\underline{X}^{[d]}Y-P(\underline{X},Z)$ is irreducible $k[\underline{X},Y,Z]$. Hence, $R$ is a  domain. We view $R$ as a subring of $A(L_{[d],[e]})$ by identifying the images of $\underline{X}, Y, Z$ in $R$ with $\underline{x}, y, z$ in $A(L_{[d],[e]})$, whence $L=R[T]/(\underline{x}^{[e]}T- Q(\underline{x},y,z))$.

	Now $R/(\underline{x})\cong \frac{k[Z]}{P(\underline{0},Z)}[Y]$. Since $Q(\underline{X}, Y, Z)$ is monic in $Y$, it follows that $Q(\underline{x},y,z)$ is not a zero-divisor on $R/(\underline{x})$. By Lemma \ref{lem: gen zero-divisor}, $\underline{x}^{[e]}$ is not a zero-divisor on $R/(Q(\underline{x}, y, z))$. Therefore, $L$ is a domain due to Lemma \ref{lem: integrity}.
\end{proof}
	
\begin{lemma}\label{lem: rigiditycrieterion}\cite[Lemma 3.3]{2013The}
	An affine domain $A$ of transcendence degree one over a field $k$ is rigid if $\mathrm{Spec}(A)$ has a singular point.
\end{lemma}
In the next two lemmas, we will construct from $A(L_{[d],[e]})$ a sequence of graded rings such that each one is the associated graded ring of its predecessor, and finally we arrive at a special graded domain $V$.

\begin{lemma}\label{lem: reducetoU}

Let $A^{(0)}:=A(L_{[d],[e]})$. Then we can construct a sequence of domains $A^{(1)},A^{(2)},\ldots,A^{(n)}$ such that for each $0\leq j\leq n-1$ there exists an admissible proper filtration on $A^{(i)}$ such that $\textup{Gr}(A^{(j)})\cong A^{(j+1)}$, $0\leq j\leq n-1$ and
\begin{equation}\label{eq: equationofU}	
	A^{(n)}=\frac{k[\underline{X}, Y, Z, T]}{(\underline{X}^{[d]}Y- P(\underline{0},Z), \underline{X}^{[e]}T- Y^{s})}.
\end{equation}
\end{lemma}

\begin{proof} For each $1\leq j\leq n$, define
\begin{equation}\label{eq: equationofL_{1}}
	 		A^{(j)}:=  \frac{k[\underline{X}, Y, Z, T]}{(\underline{X}^{[d]}Y- P(0,\ldots,0,X_{i+1},\ldots,X_{n}, Z), \underline{X}^{[e]}T- Y^{s})},
	\end{equation}
where  \begin{equation}
	 		A^{(n)}=  \frac{k[\underline{X}, Y, Z, T]}{(\underline{X}^{[d]}Y- P(0, Z), \underline{X}^{[e]}T- Y^{s})},
	\end{equation}
Note that $A^{(0)}:=A(L_{[d],[e]})$ can be seen as a subring of the $\mathbb{Z}$-graded ring
$k[\underline{x},\underline{x}^{-1},z]=\oplus_{i\in \mathbb{Z}}k[x_{2},x_{3},\ldots,x_{n},x_{2}^{-1},x_{3}^{-1},\ldots,x_{n}^{-1},z]x_{1}^{i}$. Define a proper $\mathbb{Z}$-filtration $\{A_m^{(0)}\mid m\in \mathbb{Z}\}$ on $A^{(0)}$, where \[A_m^{(0)}:=A^{(0)}\cap ( \oplus_{i \geq -m}k[x_{2},x_{3},\ldots,x_{n},x_{2}^{-1},x_{3}^{-1},\ldots,x_{n}^{-1},z]x_{1}^{i}).\]

Now we verify that $\textup{Gr}(A^{(0)})\cong A^{(1)}$. Note that \[x_{1} \in A^{(0)}_{-1}\backslash A^{(0)}_{-2},~ x_{2},\ldots, x_{n},z\in A_{0}^{(0)}\backslash A_{-1}^{(0)}, ~y\in A_{d_{1}}^{(0)}\backslash A_{d_{1}-1}^{(0)}\] and $t\in A_{(d_{1}s+e_{1})}^{(0)}\backslash A_{(d_{1}s+e_{1}-1)}^{(0)}$. By equalities $\underline{x}^{[d]}y= P(\underline{x},z)$ and $\underline{x}^{[e]}t=Q(\underline{x},y,z)$, we see that any element $q \in A^{(0)}$ can be written as
	\begin{align}\label{eq: expaninL}
		q= f_{0}(\underline{x},z)&+\sum_{\substack{ j> 0 \\ [\alpha]\in \mathbb{Z}_{\geq 0}^{n}, ~[\alpha]\ngeq [d]}}a_{j,[\alpha]}(z)\underline{x}^{[\alpha]}y^{j}\notag                                         +             \sum_{\substack{l>0 \\ [\beta]\in  \mathbb{Z}_{\geq 0}^{n},~ [\beta]\ngeq [e] }}b_{l,[\beta]}(z)\underline{x}^{[\beta]}t^{l}  \notag \\
		& + \sum_{\substack{j,l>0\\ [\gamma]\in \mathbb{Z}_{\geq 0}^{n},~[\gamma]\ngeq [d],~[\gamma]\ngeq [e]}}c_{j,l,[\gamma]}(z)\underline{x}^{[\gamma]}y^{j}t^{l},
	\end{align}
	where $f_{0}(x,z)\in k[\underline{x}, z],a_{j,[\alpha]}(z), b_{l,[\beta]}(z), c_{j,l,[\gamma]}(z)\in k[z]$. For $q\in A^{(0)}$, let $\tilde{q}$ denote the image of $q$ in $\textup{Gr}(A^{(0)})$. It follows from \eqref{eq: expaninL} that the filtration is admissible with the generating set $\{\underline{x}, y, z, t\}$. Hence, $\textup{Gr}(A^{(0)})$ is generated by $\tilde{\underline{x}},\tilde{y},\tilde{z},\tilde{t}$.
	
	Observing that, $\underline{x}^{[d]}y, P(0,x_{2},\ldots,x_{n},z) \in A_{0}^{(0)}$, $\underline{x}^{[d]}y-P(0,x_{2},\ldots,x_{n}, z)$
	$\in A_{-1}^{(0)}$, we have $\tilde{\underline{x}}^{[d]}\tilde{y}-P(0,\tilde{x_{2}},\ldots,\tilde{x_{n}},\tilde{z})= 0$ in $\textup{Gr}(A^{(0)})$. Again by $\underline{x}^{[e]}t, y^{s}\in A_{d_{1}s}^{(0)}$ and $\underline{x}^{[e]}t- y^s=Q(\underline{x}, y, z)- y^{s}\in A_{d_{1}s-1}^{(0)}$, we have $\tilde{\underline{x}}^{[e]}\tilde{t}-\tilde{y}^{s}= 0$ in $\textup{Gr}(A^{(0)})$. $\textup{Gr}(A^{(0)})$ can be seen as a subring of $\textup{Gr}(k[\underline{x},\underline{x}^{-1},z])\cong k[\underline{x},\underline{x}^{-1},z]$, and  the elements $\tilde{\underline{x}}$ and $\tilde{z}$ of $\textup{Gr}(A^{(0)})$ are algebraically independent over $k$. Since $A^{(0)}$ is a domain by Lemma \ref{lem: integrity}, we have $\textup{Gr}(A^{(0)})\cong A^{(1)}$.

By the same way, each $A^{(j)}$ can be seen as a subring of the $\mathbb{Z}$-graded ring
\[k[\underline{x},\underline{x}^{-1},z]=\oplus_{i\in \mathbb{Z}}k[x_1,\ldots,\hat{x}_{i},\ldots,x_{n},x_{1}^{-1},\ldots,\hat{x}_{i}^{-1},\ldots,x_{n}^{-1},z]x_{1}^{i},\]
 where $\hat{x}_i$ means omitting $x_i$. Define a proper $\mathbb{Z}$-filtration $\{A_m^{(j)}\mid m\in \mathbb{Z}\}$ on $A^{(j)}$, where \[A_m^{(j)}:=A^{(j)}\cap ( \oplus_{i \geq -m}k[x_1,\ldots,\hat{x}_{i},\ldots,x_{n},x_{1}^{-1},\ldots,\hat{x}_{i}^{-1},z]x_{1}^{i}).\]
By the same discussion as that of $\textup{Gr}(A^{(0)})\cong A^{(1)}$ above, one may verify that $\textup{Gr}(A^{(j)})\cong A^{(j+1)}$ for each $j$, which completes the proof.
\end{proof}

\begin{lemma}\label{lem: reducetoV}
	Write $U$ for  $A^{(n)}$ in \eqref{eq: equationofU} in Lemma \ref{lem: reducetoU}, i.e., \[U=\frac{k[\underline{X}, Y, Z, T]}{(\underline{X}^{[d]}Y- P(\underline{0},Z), \underline{X}^{[e]}T- Y^{s})}.\]
Then there exists an admissible proper filtration on $U$ such that the associated graded ring $\textup{Gr}(U)$ is isomorphic to \[V:= \frac{k[\underline{X}, Y, Z ,T]}{(\underline{X}^{[d]}Y - Z^{r}, \underline{X}^{[e]}T-Y^{s})}.\]
\end{lemma}
\begin{proof}
Let $\tilde{\underline{x}}, \tilde{y},\tilde{z},\tilde{t}$ denote the images of $\underline{X},Y , Z, T$ in $U$ respectively. See $U$  as a subring of the $\mathbb{N}$-graded ring $k[\tilde{\underline{x}},\tilde{\underline{x}}^{-1},\tilde{z}]= \oplus_{i\in \mathbb{N}}k[\tilde{\underline{x}}, \tilde{\underline{x}}^{-1}]\tilde{z}^{i}$, and define a proper $\mathbb{Z}$-filtration $\{U_{n}\}_{n\in \mathbb{Z}}$ on $U$ by \[U_{n}:= U\cap(\oplus_{i\leq n}k[\tilde{\underline{x}},\tilde{\underline{x}}^{-1}]\tilde{z}^{i}).\] Note that, $U_{n}=\emptyset$ for all $n<0, \tilde{z}\in U_{1}\backslash U_{0}, \tilde{\underline{x}}\in U_{0}\backslash U_{-1}, \tilde{y}\in U_{r}\backslash U_{r-1}$ and $\tilde{t}\in U_{rs}\backslash U_{rs-1}$. Using the relation $\tilde{\underline{x}}^{[d]}\tilde{y}=P(\underline{0}, \tilde{z})$ and $\tilde{\underline{x}}^{[e]}\tilde{t}= \tilde{y}^{s}$, we see that each element $\tilde{q}\in U$ can be written as
	\begin{equation}\label{eq: expaninU}
		\tilde{q}=\sum_{i=0}^{r-1}(\sum_{0\leq j <s} q_{ij}(\tilde{\underline{x}})\tilde{y}^{j}+ \sum_{\substack{0\leq j <s\\l>0}}h_{ijl}(\tilde{\underline{x}})\tilde{y}^{j}\tilde{t}^{l})\tilde{z}^{i},
	\end{equation}
	where $q_{ij}(\tilde{\underline{x}}), h_{ijl}(\tilde{\underline{x}})\in k[\tilde{\underline{x}}]$.

Let $\overline{U}$ denote the graded ring $\textup{Gr}(U)=\oplus_{n\in \mathbb{Z}}(U_{n}/U_{n-1})$ with respect to the above filtration. For $\tilde{q}\in U$, let $\overline{q}$ denote the image of $\tilde{q}$ in $\overline{U}$. It follows from \eqref{eq: expaninU} that the filtration just defined on $U$ is admissible with the generating set $\{\tilde{\underline{x}}, \tilde{y}, \tilde{z}, \tilde{t}\}$. So $\overline{U}$ is generated  by $\overline{\underline{x}}, \overline{y}, \overline{z}, \overline{t}$.
	
	Because $\tilde{\underline{x}}^{[d]}\tilde{y}, \tilde{z}^{r}\in U_{r}$ and $\tilde{\underline{x}}^{[d]}\tilde{y}-\tilde{z}^{r}= P(\underline{0}, \tilde{z})-\tilde{z}^{r}\in U_{r-1}$, we  have $\overline{\underline{x}}^{[d]}\overline{y}- \overline{z}^{r}= 0$ in $\overline{U}$. Also, by $\tilde{\underline{x}}^{[e]}\tilde{t},\tilde{y}^{s}\in U_{rs}$ and $\tilde{\underline{x}}^{[e]}\tilde{t}- \tilde{y}^{s}=0$ in $U$, we obtain that $\overline{\underline{x}}^{[e]}\overline{t}- \overline{y}^{s}=0$ in $\overline{U}$.  See $\overline{U}$ as a subring of $\textup{Gr}(k[\tilde{\underline{x}},\tilde{\underline{x}}^{-1},\tilde{z}])\cong k[\tilde{\underline{x}},\tilde{\underline{x}}^{-1},\tilde{z}]$, we observe that the elements $\overline{\underline{x}}$ and $\overline{z}$ of $\overline{U}$ are algebraically independent over $k$. By Lemma \ref{lem: integrity}, $V$ is a domain. Therefore, $\textup{Gr}(U)= \overline{U}\cong V$.
\end{proof}

\begin{lemma}\label{lem: invariantinV}
	Let $V$ be the domain defined by
	\begin{equation}
		\frac{k[\underline{X}, Y, Z, T]}{(\underline{X}^{[d]}Y-Z^{r}, \underline{X}^{[e]}T-Y^{s})}, \notag
	\end{equation}
where $[d], [e]\in \mathbb{Z}_{\geq 1}^{n}$, $r\geq 2$ and $s\geq 2$.
Denote by $\overline{\underline{x}}, \overline{y}, \overline{z}, \overline{t}$  the images of $\underline{X}, Y, Z, T$ in $V$ respectively. Consider $V= \oplus_{i\in \mathbb{Z}}V_{i}$ as a graded subring of $k[\overline{\underline{x}}, \overline{\underline{x}}^{-1}, \overline{z}]$ with \begin{equation}
		V_{i}= V\cap k[\overline{\underline{x}}, \overline{\underline{x}}^{-1}]\overline{z}^{i}\ \text{for each }\  i\geq 0 \ \text{and }\ V_{i}=0 \ \text{for }\ i<0.\notag
	\end{equation}
	Then $V^{\phi}\subseteq k[\overline{\underline{x}}]$ for any nontrivial homogeneous exponential map $\phi $ on the graded ring $V$.
\end{lemma}

\begin{proof}
	Let $\phi$ be a homogeneous exponential map on $V$. This grading defines a degree function on $V$ with $\deg\overline{\underline{x}}=0, \deg \overline{z}=1, \deg \overline{y}=r$ and $\deg\overline{t}=rs$. Let
	\[R=\frac{k[\underline{X}, Y, T]}{(\underline{X}^{[e]}T-Y^{s})}.\]
	We can see $R$ as a subring of $V$ by identifying the images of $\underline{X}, Y\ \text{and }\ T$ in $R$ with $\overline{\underline{x}}, \overline{y},\overline{t}$ in $V$. Obviously, $R$ is embedded in $\oplus_{i\in r\mathbb{Z}}V_{i}$. Below we show that $V^{\phi}\subseteq R$.

Note that each element $f\in V$ can be uniquely written as
	$f=\sum_{i=0}^{r-1}f_{i}\overline{z}^{i},$ where $f_{i}\in R$. For if,
	\[\deg (f_{i}\overline{z}^{i})=\deg(f_{j}\overline{z}^{j})\ \text{for some }\ 0\leq i, j \leq r-1,\] then $\deg f_{i}+ i=\deg f_{j}+j$. Since $f_{i}, f_{j}\in R$, we have
	$\deg f_{i}-\deg f_{j}\equiv 0  \mod\  r,$ i.e., $i-j\equiv 0  \mod r$, which implies that $i=j$.

	Suppose  that $V^{\phi}\nsubseteq R$. Then as $V^{\phi} $ is a graded subring of $V$, we have $ f_{i}\overline{z}^{i}\in V^{\phi}$ for some $f_{i}\in R$ with $i>0$. By  Lemma \ref{lem:expproperty} (1), $\overline{z}\in V^{\phi}$, and using the equalities $\overline{\underline{x}}^{[d]}\overline{y}=\overline{z}^{r}$ and $\overline{\underline{x}}^{[e]}\overline{t}=\overline{y}^{s}$, we see that $\overline{\underline{x}}, \overline{y}, \overline{t}\in V^{\phi}$, i.e., $\phi$ is trivial, which is a contradiction. Hence $V^{\phi}\subseteq R$.
	
	Below we show that $V^{\phi}\subseteq k[\overline{\underline{x}}]$. Any element $q\in R$ can be written as
	
	\begin{equation}
		q=\sum_{0\leq i<s} q_{i}(\overline{\underline{x}}) \overline{y}^{i}  + \sum_{\substack{i>0\\ 0\leq j<s}} q_{ij}(\overline{\underline{x}}) \overline{t}^{i}\overline{y}^{j},\notag		
	\end{equation}
	where $q_{i}, q_{ij}\in k[\overline{\underline{x}}]$. Note that
	
	\begin{equation}	
		\begin{split}
				\deg(q_{i}(\overline{\underline{x}})\overline{y}^{i})&= ir<sr \ \text{if}\ i<s\ \text{and }\ \\ \deg(q_{ij}(\overline{\underline{x}})\overline{t}^{i}\overline{y}^{j})&=(irs+jr) \
			\text{if} \ i>0, 0\leq j <s.
		\end{split}		
		\end{equation}
	Thus a homogeneous element of $V$ in $R$ is of the form $q_{i}(\overline{\underline{x}})\overline{y}^{i}$ for some $0\leq i<s$ or $q_{ij}(\overline{\underline{x}})\overline{t}^{i}\overline{y}^{j}$ for some $i>0\ \text{and }\ 0\leq j <r$. As $V^{\phi }$ is a graded subring of $V$, we have either $q_{i}(\overline{\underline{x}})\overline{y}^{i}\in V^{\phi}$ for some $0\leq i<s\ \text{or}\ q_{ij}(\overline{\underline{x}})\overline{t}^{i}\overline{y}^{j}\in V^{\phi}$ for some $i>0\ \text{and }\ 0\leq j<r$.

Suppose that $q_{ij}(\overline{\underline{x}})\overline{t}^{i}\overline{y}^{j}\in V^{\phi}$ for some $i>0\ \text{and }\ 0\leq j<r$. Then, $\overline{t}\in V^{\phi}$. As, by Lemma \ref{lem:expproperty} (4), $\textup{trdeg}_kV^{\phi}=\textup{trdeg}_k V-1=(n+1)-1=n$, we may choose $n-1$ algebraically independent homogeneous elements in $V^{\phi}$, say $b_{1},\ldots, b_{n-1}$.
	Note that it is no harm to let $b_{1}=h_{1}(\underline{x}, y),\ldots, b_{n-1}=h_{n-1}(\underline{x}, y)$ all containing no constant terms. Then the ring
\[\dot{A}:=\frac{k(t, b_{1},\ldots, b_{n-1})[\underline{x}, y, z, t]}{(b_{1}-h_{1}(\underline{x},y),\ldots,~ b_{n-1}-h_{n-1}(\underline{x},y),~ \underline{x}^{[d]}y-z^{r},~ \underline{x}^{[e]}t-y^{s})}\] is of one-dimensional  over $k(t,b_1,\ldots, b_{n-1})$ and $\phi$ can be extended to $\dot{A}$. But it's easily to see that the ring $\dot{A}$ is singular at the zero point, 
 and thus $\dot{A}$ is a rigid ring due to Lemma \ref{lem: rigiditycrieterion}, and we have the extension $\dot{\phi}$ of $\phi$ from $V$ to $\dot{A}$ is trivial, which is a contradiction. 
	
	Hence $q_{i}(\overline{\underline{x}})\overline{y}^{i}\in V^{\phi}$ for some $0\leq i<s$. If $i>0$, then $\overline{y}\in V^{\phi}$ since $V^{\phi}$ is factorially closed in $V$. Using the relations $\overline{\underline{x}}^{[e]}\overline{t}=\overline{y}^{s}$ and $\overline{\underline{x}}^{[d]}\overline{y}= \overline{z}^{r}$, we get, $\overline{\underline{x}}, \overline{z},\overline{t}\in V^{\phi}$, i.e., $\phi$ is trivial, which is a contradiction. Therefore, $i=0$ and $V^{\phi}\subseteq k[\overline{\underline{x}}]$.
\end{proof}

\begin{lemma}\label{lem: examplemap}
There exists a nontrivial exponential map $\phi$ on $A(L_{[d],[e]})$ such that $A(L_{[d],[e]})^{\phi}= k[\underline{x}]$.
\end{lemma}

\begin{proof} Let $A:=A(L_{[d],[e]})$. Consider the $k$-algebra homomorphism $\phi: A \rightarrow A[W]$ defined by,
	\begin{align}
		&\phi(x_{i})=x_{i}, \ \text{for}\ i\in \{1,\ldots, n\}\\ \notag
		&\phi(z)=z+\underline{x}^{([d]+[e])}W,\\ \notag
		&\phi(y)=\frac{P(\underline{x}, z+\underline{x}^{([d]+[e])}W)}{\underline{x}^{[d]}}=y+W\theta(\underline{x}, z, W),\\ \notag
		&\phi(t)=\frac{Q(\underline{x}, y+W\theta(\underline{x},z,W), z+\underline{x}^{([d]+[e])}W)}{\underline{x}^{[e]}}=t+W\mu(\underline{x}, y, z, W),
	\end{align}
	where $\theta(\underline{x}, z, W)\in k[\underline{x},z, W]$ and $\mu(\underline{x}, y, z, W)\in k[\underline{x}, y, z, W] $. One may verify that $\phi$ is an exponential map on $A$. Clearly $k[\underline{x}]\subseteq A^{\phi}$. Since $k[\underline{x}]$ is algebraically closed in $A$ and $\textup{trdeg}_{k[\underline{x}]}A=1$, we have $A^{\phi}=k[\underline{x}]$.
\end{proof}

\begin{theorem}\label{thm: ML-invariantofL}
	Let $A(L_{[d],[e]})$ be as in \eqref{eq: A} and suppose that the parameters $r, s$ and $e$ in $A(L_{[d],[e]})$ satisfy the conditions of \eqref{lem: invariantinV}. Then $\mathrm{ML}(A(L_{[d],[e]}))=k[\underline{x}]$.
\end{theorem}

\begin{proof} Let $A^{(0)}:=A(L_{[d],[e]})$. We will show that $(A^{(0)})^{\phi_0}=k[\underline{x}]$ for any nontrivial exponential map $\phi_0$ of $A^{(0)}$.
By Lemma \ref{lem: reducetoU}, there exists a sequence of domains $A^{(1)},A^{(2)},\ldots,A^{(n)}$ such that for each $0\leq j\leq n-1$ there exists an admissible proper filtration on $A^{(i)}$ such that $\textup{Gr}(A^{(j)})\cong A^{(j+1)}$, $0\leq j\leq n-1$ and
\begin{equation}
	A^{(n)}=\frac{k[\underline{X}, Y, Z, T]}{(\underline{X}^{[d]}Y- P(\underline{0},Z), \underline{X}^{[e]}T- Y^{s})}.
\end{equation}
Let $A^{(n+1)}:=V=\frac{k[\underline{X}, Y, Z ,T]}{(\underline{X}^{[d]}Y - Z^{r}, \underline{X}^{[e]}T-Y^{s})}.$
By Lemma \ref{lem: reducetoV}, there exists an admissible proper filtration on $A^{(n)}$ such that
\[\textup{Gr}(A^{(n)})\cong A^{(n+1)}=V.\]
For each $0\leq j\leq n$, denote by $\rho_j: A^{(j)}\rightarrow \textup{Gr}(A^{(j)})=A^{(j+1)}$ the canonical map.

Note that $\phi_0$ induces a nontrivial exponential map $\phi_1$ on $A^{(1)}$ such that $\rho_0((A^{(0)})^{\phi_0}) \subseteq (A^{(1)})^{\phi_1}$. Let $p\in (A^{(0)})^{\phi_0}$. Replacing $p$ by $p-\lambda$ for some $\lambda\in k^*$, we may assume that $\rho_0(p)\notin k$. In the same way, $\phi_1$ induces a nontrivial exponential map $\phi_2$ on $A^{(2)}$ such that $\rho_1((A^{(1)})^{\phi_1})\subseteq (A^{(2)})^{\phi_2}$. In turn, we obtain nontrivial exponential maps $\phi_j$ on $A^{(j)}$, $j=0,1,\ldots,n$ such that $\rho_j((A^{(j)})^{\phi_j})\subseteq (A^{(j+1)})^{\phi_{j+1}}$, and $\rho_n\rho_{n-1}\cdots\rho_0(p)\in V^{\phi_{n+1}}$.

By Lemma \ref{lem: invariantinV}, $V^{\phi_{n+1}}\subseteq k[\overline{\underline{x}}]$, where $\overline{\underline{x}}$ denotes the image of $\underline{x}$ in $V$. Hence $\rho_n\rho_{n-1}\cdots\rho_0(p)\in k[\overline{\underline{x}}]$. It follows from the filtration defined on $A^ {(n)} $ in Lemma \ref{lem: reducetoV} and the equality \eqref{eq: expaninU} that $\rho_{n-1}\cdots\rho_0(p)\in k[\tilde{\underline{x}}]$, where $\tilde{\underline{x}}$ denotes the image of $\underline{x}$ in $A^{(n)}$. Again from the filtration defined on $A^{(n-1)}$ in Lemma \ref{lem: reducetoU}, it follows that $\rho_{n-2}\cdots\rho_0(p)\in k[\tilde{\underline{x}}]\subseteq A^{(n-1)}$, where $\tilde{\underline{x}}$ denotes the image of $\underline{x}$ in $A^{(n-1)}$. By this way, after finite steps, we  get that \[(A^{(0)})^{\phi_0}\subseteq k[\underline{x}].\] Since $(A^{(0)})^{\phi_0}$ is factorially closed in $A^{(0)}$ and is of transcendence degree $n$ over $k$, we have $(A^{(0)})^{\phi_0}=k[\underline{x}]$ for any nontrivial exponential map $\phi_0$ of $A^{(0)}$. And by Lemma \ref{lem: examplemap}, there does exist a nontrivial exponential map $\phi_0$ of $A^{(0)}$, and thus $\mathrm{ML}(A^{(0)})=k[\underline{x}]$.
\end{proof}

	\section{Locally nilpotent derivations of double Danielewski varieties}
Throughout this section, $k$ will denote a field of characteristic zero.
Similar to \cite{gupta2020locally}, we give a characterization of locally nilpotent derivations of double Danielewski varieties.
\begin{theorem}\label{thm: computeof LND}
Consider the coordinate ring of a double Danielewski variety
		\begin{equation}
		A(L_{[d],[e]}):=\frac{k[\underline{X},Y,Z, T]}{(\underline{X}^{[d]}Y - P(\underline{X},Z), \underline{X}^{[e]}T - Q(\underline{X},Y,Z))}\notag
	\end{equation} where $[d], [e]\in \mathbb{Z}_{\geq 1}^{n}$.
	Then any locally nilpotent derivation of $A(L_{[d],[e]})$ is of the form $\delta D$, where $\delta\in k[\underline{x}]$ and $D$ is given	by
	\begin{align}
		&D(x_{i})=0, \text{~for~all}~i\in \{1,\ldots, n\}, \notag \\
		&D(z) =\underline{x}^{[d]+[e]},  \notag        \\
		&D(y) = \frac{\partial P(\underline{x},z)}{\partial z}\underline{x}^{[e]}, \notag   \\
		&D(t) =\frac{\partial Q(\underline{x},y, z)}{\partial y}\frac{\partial P(\underline{x},z)}{\partial z}+ \frac{\partial Q(\underline{x}, y, z)}{\partial z}\underline{x}^{[d]}.	\notag
	\end{align}
\end{theorem}	

\begin{proof} Let $A:=A(L_{[d],[e]})$ and let $D$ be any locally nilpotent derivation of $A$. By Theorem \ref{thm: ML-invariantofL}, $\mathrm{ML}(A)=k[\underline{x}]$. Since $D(\underline{x})=0$, $D$ can be extended to a $k$-linear locally nilpotent derivation of $A[x^{-1}]= k[\underline{x},\underline{x}^{-1},z]$. Thus $D(z)\in k[\underline{x},\underline{x}^{-1}]\cap A=k[\underline{x}]$. Let $D(z)=r(\underline{x})$, where $r(\underline{x})\in k[\underline{x}]$.

By
	$D(\underline{x}^{[d]}y-P(\underline{x},z))=0$, we have \[\underline{x}^{[d]}D(y)-\frac{\partial P(\underline{x},z)}{\partial z}D(z)=0,\] and by $D(\underline{x}^{[e]}t-Q(\underline{x}, y, z))=0$, we get
		\[\underline{x}^{[e]}D(t)-\frac{\partial Q(\underline{x}, y, z)}{\partial y}D(y)- \frac{\partial Q(\underline{x}, y, z)}{\partial z}D(z)=0.\] Hence
	\begin{align}
		D(z)&=r(\underline{x}),\notag\\
		D(y)&= \frac{\frac{\partial P(\underline{x}, z)}{\partial z}r(\underline{x})}{\underline{x}^{[d]}}, \notag\\
		D(t)&=\frac{\frac{\partial Q(\underline{x}, y , z )}{\partial y}\frac{\partial P(\underline{x}, z)}{\partial z}r(\underline{x})+ \frac{\partial Q(\underline{x}, y, z)}{\partial z}\underline{x}^{[d]}r(\underline{x})}{\underline{x}^{[d]+[e]}}.\notag
	\end{align}
	So it remains to show that $r(\underline{x})\in \underline{x}^{[d]+[e]}k[\underline{x}]$.	
	
	From $D(y)\in A$, we have
	\[\frac{\partial P(\underline{x},z)}{\partial z}r(\underline{x})\in \underline{x}^{[d]}A \cap k[\underline{x}, z]= (\underline{x}^{[d]}, P(\underline{x}, z))k[\underline{x}, z].\]
	Hence
	\begin{equation}\label{eq: numeratorofDyexpansion}
		\frac{\partial P(\underline{x}, z)}{\partial z}r(\underline{x}) = \underline{x}^{[d]}g_{1}(\underline{x}, z) + P(\underline{x}, z)g_{2}(\underline{x}, z)
	\end{equation}
	for some $g_{1}, g_{2}\in k^{[2]}$. We note that $k[\underline{x},z]$ is isomorphic to a polynomial ring in two variables.
	
For every $i\in \{1,\ldots,n\}$, let $l_{i} \geq 0$ be such that $x_{i}^{l_{i}} \mid r(\underline{x})$ but $\underline{x}^{l_{i}+1} \nmid r(\underline{x})$. We claim that $l_{i}\geq d_{i}$. Suppose by contrary that for some $j\in\{1,\ldots,n\}$, $l_{j}<d_{j}$, then $x_{j}^{l_{j}} \mid P(\underline{x},z)g_{2}(\underline{x},z)$. Since $P(\underline{x},z) $ is monic in $z$ we have $x_{j}^{l_{j}}\mid g_{2}(\underline{x},z)$. Let $g_{2}(\underline{x},z)= x_{j}^{l_{j}}\tilde{g}_{2}(\underline{x},z)$ and $r(\underline{x})= x_{j}^{l_{j}}\tilde{r}(\underline{x})$. Then from \eqref{eq: numeratorofDyexpansion}, we have
	\begin{equation}\label{eq: refinenumeratorofDyexpansion}
		\frac{\partial P(\underline{x},z)}{\partial z}\tilde{r}(\underline{x})= \underline{x}^{[d]-(0,\ldots,l_{j},\ldots,0)}g_{1}(\underline{x},z) + P(\underline{x},z)\tilde{g}_{2}(\underline{x},z).
	\end{equation} 	
	Putting $x_{j}=0$ in \eqref{eq: refinenumeratorofDyexpansion} we have
	\[\frac{\partial P(\underline{x},z)}{\partial z}\vert_{x_{j}=0} \tilde{r}(\underline{x})|_{x_{j}=0}=P(\underline{x},z)|_{x_{j}=0}\tilde{g_{2}}(\underline{x},z)|_{x_{j}=0}.\]
	Since both $\frac{P(\underline{x},z)}{\partial z}\vert_{x_{j}=0} $ and $P(\underline{x},z)|_{x_{j}=0}$ are non-zero and $\deg_{z}\frac{P(\underline{x},z)}{\partial z}\vert_{x_{j}=0} <\deg_{z}P(\underline{x},z)|_{x_{j}=0}$, we get $\tilde{r}(\underline{x})|_{x_{j}=0}=0$, which contradicts that $x_{j}^{l_{j}+1}\nmid r(\underline{x})$. Hence $l_{i}\geq d_{i},$ for any $i\in\{1,\ldots,n\}$. We write $[l]=(l_{1},\ldots,l_{n}).$
	
Write $r(\underline{x})=\underline{x}^{[l]}\overline{r}(\underline{x})$ and let $\alpha(\underline{x})= \underline{x}^{[l]-[d]}\overline{r}(\underline{x})$, i.e., $r(\underline{x})=\underline{x}^{[d]}\alpha(\underline{x})$. Now $D(t)\in A$ implies that
	\[\big(\frac{\partial Q(\underline{x},y, z)}{\partial y}\frac{\partial P(\underline{x},z)}{\partial z}+\frac{\partial Q(\underline{x},y,z)}{\partial z}\underline{x}^{[d]}\big)r(\underline{x})\in \underline{x}^{[d]+[e]}A,\]
	i.e., \[\big(\frac{\partial Q(\underline{x},y,z)}{\partial y}\frac{\partial P(\underline{x},z)}{\partial z}+\frac{\partial Q(\underline{x},y,z)}{\partial z}\underline{x}^{[d]}\big)\alpha(\underline{x})\in \underline{x}^{[e]}A.\]
	Now \begin{align*}&\big(\frac{\partial Q(\underline{x},y,z)}{\partial y}\frac{\partial P(\underline{x},z)}{\partial z}+\frac{\partial Q(\underline{x},y,z)}{\partial z}\underline{x}^{[d]}\big) \alpha(\underline{x})\\&\in \underline{x}^{[e]}A\cap k[\underline{x},y, z]= (\underline{x}^{[e]},
\underline{x}^{[d]}y-P(\underline{x},z), Q(\underline{x},y,z))k[\underline{x},y,z].
\end{align*}
Hence 	
			\begin{equation}\label{eq: numeratorofDtexpansion}
				\begin{split}
						&\big(\frac{\partial Q(\underline{X},Y,Z)}{\partial Y}\frac{\partial P(\underline{X},Z)}{\partial Z}+ \frac{\partial Q(\underline{X},Y, Z)}{\partial Z}\underline{X}^{[d]}\big)\alpha(\underline{X})\\&= \underline{X}^{[e]}\alpha_{1}+ (\underline{X}^{[d]}Y-P(\underline{X},Z))\alpha_{2}
					+ Q(\underline{X},Y, Z) \alpha_{3}
				\end{split}		
		\end{equation}
for some $\alpha_{1},\alpha_{2},\alpha_{3}\in k[\underline{X},Y,Z]$.

Since $r(\underline{x})=\underline{x}^{[d]}\alpha(\underline{x})$ and $\alpha(\underline{x})= \underline{x}^{[l]-[d]}\overline{r}(\underline{x})$, to show that $r(\underline{x})\in \underline{x}^{[d]+[e]}k[\underline{x}]$, it suffices to verify the following claim, and which will complete the proof of the theorem.

\textbf{Claim: }$[l]-[d]\geq [e]$.

We divide $\alpha_{1},\alpha_{2}$ by $Q$ and let $\beta_{1}, \beta_{2}$ be the  remainders respectively. Then $\deg_{Y}
\beta_{1}<s, \deg_{Y}$
$\beta_{2}<s$ and
	\begin{equation}\label{eq: numeratorofDtrightexpansion}
		\begin{split}
				&\underline{X}^{[e]}\alpha_{1}+ (\underline{X}^{[d]}Y-P(\underline{X},Z))\alpha_{2} +Q(\underline{X},Y,Z)\alpha_{3}\\&= \underline{X}^{[e]}\beta_{1}+(\underline{X}^{[d]}Y-P(\underline{X},Z))\beta_{2} +Q(\underline{X},Y,Z)\beta_{3}
		\end{split}	
	\end{equation}	
	for some $\beta_{3}\in k[\underline{X},Y,Z]$. Thus from \eqref{eq: numeratorofDtexpansion} and \eqref{eq: numeratorofDtrightexpansion}, we have
	\begin{equation}\label{eq: refinenumeratorofDtexpansion}
		\begin{split}
				&\big(\frac{\partial Q(\underline{X},Y,Z)}{\partial Y}\frac{\partial P(\underline{X},Z)}{\partial Z}+ \frac{\partial Q(\underline{X},Y, Z)}{\partial Z}\underline{X}^{[d]}\big)\alpha(\underline{X})\\&=\underline{X}^{[e]}\beta_{1}+(\underline{X}^{[d]}Y-P(\underline{X},Z))\beta_{2} +Q(\underline{X},Y,Z)\beta_{3}.
		\end{split}
	\end{equation}
To show the claim, we consider the following two subcases.
	
\textbf{Case 1:} $\deg_{Y}(Y\beta_{2})<s$.

In this case, since $\deg_{Y}Q\beta_{3} \geq s$ and
	\[\deg_{Y}\{(\frac{\partial Q(\underline{X},Y,Z)}{\partial Y}\frac{\partial P(\underline{X},Z)}{\partial Z}+ \frac{\partial Q(\underline{X},Y, Z)}{\partial Z}\underline{X}^{[d]})\alpha(\underline{X})-\underline{X}^{[e]}\beta_{1}\]
	\[-(\underline{X}^{[d]}Y-P(\underline{X},Z))\beta_{2}\}<s,\]
we have $\beta_{3}=0$ (since $Q(\underline{X},Y,Z)$ is nonzero), i.e.,
	\begin{equation}\label{eq: refinenumeratorofDtexpansion2}
		\begin{split}
				&\big(\frac{\partial Q(\underline{X},Y,Z)}{\partial Y}\frac{\partial P(\underline{X},Z)}{\partial Z}+ \frac{\partial Q(\underline{X},Y, Z)}{\partial Z}\underline{X}^{[d]}\big)\alpha(X)\\&=\underline{X}^{[e]}\beta_{1}+(\underline{X}^{[d]}Y-P(\underline{X},Z))\beta_{2}.	
		\end{split}
	\end{equation}
	Suppose that $l_{j}-d_{j}<e_{j}$ for some $j\in \{1,\ldots,n\}$. Since $\underline{x}^{[l]-[d]}$ divides the left side, we have
	\[\underline{X}^{[l]-[d]}\mid (\underline{X}^{[d]}Y- P(\underline{X},Z)\beta_{2}),\] and noticing that $P(\underline{X},Z)$ is monic in $Z$, we have then
	\[\underline{X}^{[l]}\mid \beta_{2}.\]
	Let $\beta_{2}= \underline{X}^{[l]-[d]}\tilde{\beta_{2}}$. Then from \eqref{eq: refinenumeratorofDtexpansion2} we have,
	\begin{equation}\label{eq: refinenumeratorofDtexpansion3}
		\begin{split}
				&\big(\frac{\partial Q(\underline{X},Y,Z)}{\partial Y}\frac{\partial P(\underline{X},Z)}{\partial Z}+ \frac{\partial Q(\underline{X},Y, Z)}{\partial Z}\underline{X}^{[d]}\big)\overline{r}(\underline{X})\\&= \underline{X}^{[e]-([l]-[d])}\beta_{1} +(\underline{X}^{[d]}Y-P(\underline{X},Z))\tilde{\beta_{2}}.
		\end{split}
	\end{equation}
	Putting $X_{j}=0$ in \eqref{eq: refinenumeratorofDtexpansion3} and comparing the leading coefficient of $Y$, we get \[\frac{\partial P(\underline{X},Z)}{\partial Z}\vert_{X_{j}=0}\overline{r}(\underline{X})|_{X_{j}=0}=P(\underline{X},z)|_{X_{j}=0}\tau\  \text{for some }\ \tau\in k[Z],\]
	which is a contradiction as both $\frac{\partial P(\underline{X},Z)}{\partial Z}\vert_{X_{j}=0} $ and $P(\underline{X},Z)|_{X_{j}=0}$ are non-zero polynomials and $\deg_{Z}P(\underline{X},z)|_{X_{j}=0}>\deg_{Z}\frac{\partial P(\underline{X},Z)}{\partial Z}\vert_{X_{j}=0} $. Hence $[l]-[d]\geq[e]$, i.e., the claim holds in Case 1.

\textbf{Case  2:} $\deg_{Y}(Y\beta_{2})=s$. In this case, let
	\begin{equation}\label{eq: Ybeta2expansion}
		Y\beta_{2}=c_{0}Q(\underline{X}, Y, Z)+ \gamma,
	\end{equation}
	where $c_{0}\in k[X, Z]$ is the leading coefficient of $Y$ in $\beta_{2}$ and $\gamma \in k[\underline{X},Y,Z]$. Note that
	\[\deg_{Y}(Y\beta_{2}-c_{0}Q(\underline{X},Y,Z))=\deg_{Y}\gamma<s.\]
	Then the equalities \eqref{eq: refinenumeratorofDtexpansion} and \eqref{eq: Ybeta2expansion} give
	\begin{align*}&\big(\frac{\partial Q(\underline{X},Y,Z)}{\partial Y}\frac{\partial P(\underline{X},Z)}{\partial Z}+ \frac{\partial Q(\underline{X},Y, Z)}{\partial Z}\underline{X}^{[d]}\big)\alpha(X)\\&=\underline{X}^{[e]}\beta_{1}+(\underline{X}^{[d]}\gamma-P(\underline{X},Z))\beta_{2}+Q(\underline{X},Y,Z)(\beta_{3} + c_{0}\underline{X}^{[d]}).
\end{align*}
	Since
	\[\deg_{Y}\{\big(\frac{\partial Q(\underline{X},Y,Z)}{\partial Y}\frac{\partial P(\underline{X},Z)}{\partial Z}+ \frac{\partial Q(\underline{X},Y, Z)}{\partial Z}\underline{X}^{[d]}\big)\alpha(X)-\underline{X}^{[e]}\beta_{1}\]\[-(\underline{X}^{[d]}\gamma-P(\underline{X},Z))\beta_{2}\}<s\] and \[\deg_{Y}(\beta_{3}+c_{0}\underline{X}^{[d]})Q\geq s,\]
	we have $(\beta_{3}+c_{0}\underline{X}^{[d]})=0$ (since $Q(\underline{X},Y,Z) $ is non-zero), and thus\[\big(\frac{\partial Q(\underline{X},Y,Z)}{\partial Y}\frac{\partial P(\underline{X},Z)}{\partial Z}+ \frac{\partial Q(\underline{X},Y, Z)}{\partial Z}\underline{X}^{[d]}\big)\underline{X}^{[l]-[d]}\overline{r}(\underline{X})\]\[=\underline{X}^{[e]}\beta_{1}+\underline{X}^{[d]}\gamma-P(\underline{X},Z)\beta_{2}.\]
For every $i\in\{1,\ldots,n\}$, let $\mu_{i}\geq 0$ be such that $X_{i}^{\mu_{i}}\mid \beta_{2}$ and $X_{i}^{\mu_{i}+1}\nmid \beta_{2}$. Let $\beta_{2}=X_{i}^{\mu_{i}}\tilde{\beta_{2}}$. Then $X_{i}^{\mu_{i}}\mid c_{0}$ as $c_{0}$ is the leading coefficient of $Y$ in $\beta_{2}$. Now $X_{i}^{\mu_{i}}\mid \gamma$ as $\gamma=(Y\beta_{2}-c_{0}Q(\underline{X},Y,Z))$. Write $\mu =(\mu_{1},\ldots, \mu_{n})$. Let $\gamma=X_{i}^{\mu_{i}}\tilde{\gamma}$. Then
	\begin{equation}\label{eq: refinenumeratorofDtexpansionincase2}
		\begin{split}
				&\big(\frac{\partial Q(\underline{X},Y,Z)}{\partial Y}\frac{\partial P(\underline{X},Z)}{\partial Z}+
			\frac{\partial Q(\underline{X},Y, Z)}{\partial Z}\underline{X}^{[d]}\big)\underline{X}^{[l]-[d]}\overline{r}(\underline{X})\\&=\underline{X}^{[e]}\beta_{1}+ \underline{X}^{[d]+(0,\ldots,\mu_{i},\ldots,0)}\tilde{\gamma}-P(\underline{X},Z)X_{i}^{\mu_{i}}\tilde{\beta_{2}}.
		\end{split}
	\end{equation}
	Now if $\mu_{j} <e_{j}$ for some $j\in \{1,\ldots,n\}$, then the least $X_{j}$-degree term of the right side of the equation is $-X_{j}^{\mu_{j}}\tilde{\beta_{2}}(\underline{X},Y,Z)|_{X_{j}=0}P(\underline{X},Z)|_{X_{j}=0}$ and that of the left side is $\frac{\partial Q(\underline{X},Y,Z)}{\partial Y}|_{X_{j}=0}\frac{\partial P(\underline{X},Z)}{\partial Z}|_{X_{j}=0}$
	$\overline{r}(\underline{X})|_{X_{j}=0}\underline{X}^{[l]-[d]}$.
	
	Hence $\mu_{j}=l_{j}-d_{j}$ and \begin{align*}&\tilde{\beta_{2}}(\underline{X},Y,Z)|_{X_{j}=0}P(\underline{X},Z)|_{X_{j}=0}\\&=\frac{\partial Q}{\partial Y}(\underline{X}, Y,Z)|_{X_{j}=0}\frac{\partial P}{\partial Z}(\underline{X}, Z)|_{X_{j}=0}\overline{r}(\underline{X})|_{X_{j}=0}\underline{X}^{(l_{1}-d_{1},\ldots,\mu_{j}, \ldots, l_{n}-d_{n})}.
\end{align*}
	Comparing the coefficients of the leading $Y$-degree from both sides, we can see that this is not possible. Hence $[\mu]\geq [e]$ and so the left side of \eqref{eq: refinenumeratorofDtexpansionincase2} is divisible by $\underline{X}^{[e]}$ and this is possible only if $[l]-[d]\geq [e]$.

Thus the claim is proved, which completes the proof of the theorem.
\end{proof}

\end{document}